\documentclass[10pt,a4paper]{article}
\usepackage[latin1]{inputenc}
\usepackage{amsmath,footmisc}
\usepackage{amsfonts}
\usepackage{amssymb}
\usepackage{amsthm,color}
\usepackage{graphicx}
\usepackage[margin=1.25in]{geometry}
\usepackage[shortlabels]{enumitem}
\newtheorem{theorem}{Theorem}
\newtheorem{corollary}[theorem]{Corollary}
\newtheorem{proposition}[theorem]{Proposition}
\author{Caroline L. Wormell\thanks{Mathematical Sciences Institute, The Australian National University. email: {\sf caroline.wormell@anu.edu.au}}
	}
\title{Highly accurate and fine-scale estimation of equilibrium measures}

\newcommand{\C}{\mathbb{C}}

\newcommand{\E}{\mathcal{E}}
\renewcommand{\H}{\mathcal{H}}
\newcommand{\B}{\mathcal{B}}
\renewcommand{\L}{\mathcal{L}}
\renewcommand{\P}{\mathcal{P}}
\renewcommand{\d}{\mathrm{d}}

\newcommand{\bE}{\mathbb{E}}
\newcommand{\bP}{\mathbb{P}}

\newcommand{\h}{h}

\renewcommand{\j}{^{j}}

\DeclareMathOperator{\Lip}{Lip}

\begin{document}
	\maketitle
	\begin{abstract}
		Equilibrium measures are special invariant measures of chaotic dynamical systems and iterated function systems, commonly studied as salient examples of fractal measures. While useful analytic expressions are rare, computational exploration of these measures can yield useful insight, in particular in studying their Fourier decay. In this note we present simple, efficient computational methods to obtain weak estimates of equilibrium and related measures (i.e. as integrals against smooth functions) at high spatial resolution. These methods proceed via Chebyshev-Lagrange approximation of the transfer operator. One method, which estimates measures directly from spectral data, gives exponentially accurate estimates at spatial scales larger than the approximation's resolution. Another, method, which generates random point samples, has a Central Limit Theorem-style accuracy down to an exponentially small spatial resolution. This means that these measures and their Fourier decay can be studied very accurately, and at very high Fourier frequencies.
	\end{abstract}

	Iterated function systems (IFS) are families of contractions that together generate fractal sets. They have a dual notion in dynamics of full-branch expanding maps, chaotic systems that ``blow up'' small-scale behaviour to large scales in a procedural, self-similar fashion. These two notions are used in many contexts in pure mathematics: as well as genearting many important and interesting fractals, they provide various encodings of number-theoretical interest such as continued fractions. Key statistical and geometric properties of these systems such as invariant measures, periodic orbit counts and Hausdorff dimension of invariant sets are often studied by means of the so-called {\it thermodynamic formalism} \cite{Ruelle04}. This framework, rooted in an analogy with statistical mechanics, formulates these aforementioned properties in terms of optimisation problems over invariant measures of the system:
	\[ P(\varphi) := \sup_{\mu \in \mathcal{M}(f)} \int \varphi\,\d\mu + h_{\mu}(f)  \]
	where $f$ is the possibly open expanding map (or equivalently the IFS), and $h_{\mu}(f)$ is the measure-theoretic entropy of $f$ with respect to $\mu$. As a simple example, $P(0)$ yields the topological entropy of the system.
	
	Often, such optimization problems are maximised by a unique invariant measure $\mu_\varphi$. Such a measure is known as an {\it equilibrium measure}.
	\\
	
	Equilibrium measures are typically fractal measures supported on measure zero sets, and do not possess much regularity. Consequently, outside certain restricted settings (e.g. uniform measure on the middle-thirds Cantor set) they do not have analytic expressions. Computation of Gibbs measures and associated quantities such as pressures can be done with a moderate degree of accuracy for very general maps and weights via Ulam's method \cite{Froyland07}. However, the analytic full-branch maps that are very common in pure mathematical settings, Chebyshev-Lagrange polynomial interpolation is a far more effective way to discretise transfer operators \cite{Bandtlow20, Wormell19}. Chebyshev-Lagrange methods have recently been used to solve many problems, such as computing absolutely continuous invariant measures \cite{Wormell19, Wormell21}, Lyapunov exponents \cite{Wormell19, Pollicott23}, computing Hausdorff dimensions via approximating pressure functions \cite{Pollicott22
	}, and even the spectral structure of geodesic flows \cite{Bandtlow21}. In this note, we will show that Chebyshev-Lagrange discretisation of transfer operators can generate very good approximations of equilibrium measures, both by direct computation and via simulating point samples. To our knowledge, this is the first method to compute Gibbs measures that is rigorously justified. Because the outputs of the Chebyshev-Lagrange method remain very accurate to high resolutions, they are particularly effective in numerical investigation of Fourier decay of equilibrium measures.
	
	The note is structured as follows: in Section~\ref{s:Transfer} we introduce the transfer operator and related objects, Section~\ref{s:CL} we briefly describe the theory and implementation of Chebyshev-Lagrange method, in Section~\ref{s:Equilibrium} we show that integrals against equilibrium measures can be estimated very simply and effectively and describe a simple computational algorithm to achieve this; in Section~\ref{s:Markov} we show how our methods can be effectively used to obtain point samples of the measures, and in Section~\ref{s:Applications} we give some applications of our methods to computing Fourier transforms and visualising measures. Finally in Section~\ref{s:Extensions} we consider possible extensions of these methods.
	
	\section{Transfer operators}\label{s:Transfer}
	Equilibrium measures have a very nice definition in terms of a functional operator known as the {\it transfer operator}, from which some other results arise. The transfer operator with respect to an appropriately smooth potential $\varphi$ (we will leave exactly how smooth until the next section) is defined to act on suitably smooth functions $\psi$ as
	\begin{equation} \L_\varphi \psi(x) = \sum_{\iota \in I} e^{\varphi_\iota(x))} \psi(g_\iota(x))), \label{eq:Transfer} \end{equation}
	where the contractions $g_\iota: [-1,1] \circlearrowleft$ are the branches of the iterated function system (or equivalently the inverse branches of the chaotic map $f$), and the weights are given as $\varphi_\iota := \varphi \circ \circ g_\iota$. If the $g_\iota, \varphi_\iota$ are suitably regular, then there exists some Banach algebra of functions $\B$ so that $\L_\varphi$ is a bounded endomorphism on $\B$ with spectral radius $e^{P(\varphi)}$. Furthermore, $\L_\varphi$ has compactness properties so that a kind of Perron-Frobenius theorem obtains: in particular, $\L_\varphi$ has a simple isolated eigenvalue at $e^{P(\varphi)}$, with (right) {\it eigenfunction} $\h_\varphi \in \B$ and left eigendistribution $\nu_\varphi \in \B^*$: $\nu_\varphi$ is commonly known as the {\it conformal measure}. It turns out that these two objects multiply together to give the corresponding equilibrium measure $\mu_\varphi$: for any $\psi \in \B$,
	\begin{equation} \int \psi\,\d\mu = \frac{\nu_\varphi[\psi \h_\varphi]}{\nu_\varphi[\h_\varphi]}. \label{eq:EquilibriumEigenfunction}\end{equation}
	
	\section{Chebyshev-Lagrange approximation}\label{s:CL}
	
	The idea of the Chebyshev-Lagrange method is to perform polynomial interpolation of the action of the transfer operator, which we will assume without loss of generality acts on the interval $[-1,1]$. This is done at certain, well-chosen nodes: these are the {\it Chebyshev nodes of the first kind}
	\[ x_{n,N} = \cos\tfrac{2n-1}{2N} \pi,\, n = 1,\ldots, N. \] 
	We can interpolate a function at these points using the Lagrange polynomials
	\[ \ell_{n,N}(x) = \prod_{m \neq n} \frac{x - x_{m,N}}{x_{n,N} - x_{m,N}}, \]
	which by construction have the property that 
	\begin{equation} \ell_{n,N}(x_{m,N}) = \delta_{mn}. \label{eq:LagrangeProperty}\end{equation}
	
	 This means we can define the interpolation map $\P_N$ on $C^0$ functions
	\[ \P_N \psi = \sum_{n=1}^N \ell_{n,N} \psi(x_{n,N}). \]
	For computational purposes, it is worth noting that the Chebyshev-Lagrange polynomials also have the following closed form for $x = \cos \theta \neq x_{n,N}$:
	\[ \ell_{n,N}(\cos \theta) = \frac{1}{2N} \left(\frac{\sin N(\theta-\theta_{n,N})}{\tan \tfrac{1}{2}(\theta-\theta_{n,N})} + \frac{\sin N(\theta+\theta_{n,N})}{\tan \tfrac{1}{2}(\theta+\theta_{n,N})} \right) \]
	where $\theta_{n,N} = \cos^{-1} x_{n,N} = \tfrac{2n-1}{2N} \pi$.
	
	Chebyshev-Lagrange interpolation is a very effective approximation method. The accuracy of the interpolant depends on the regularity of the interpolated functions. We will assume that everything is analytic, since this is the nicest and most common setting in pure mathematical contexts.
	
	We will first describe some Banach spaces of analytic functions. For $r > 0$ let $\E_r \subset \C$ be the complex open ellipse centred at $0$ with semi-major axis $\cosh r$ and semi-minor axis $\sinh r$. These so called Bernstein ellipses enclose $[-1,1]$, and shrink onto it as $r\to 0$. Then $\H_r$ is the space of bounded analytic functions on $\E_r$ with
	\[ \| \psi \|_{\H_r} := \sup_{z \in \E_r} | \psi(z) |. \]
	These function spaces compactly embed inside each other with $\|\cdot \|_{\H_r} \leq \|\cdot \|_{\H_R}$ for $R>r$, and they obey the Banach algebra property that $\|\psi \chi \|_{\H_r} \leq \|\psi\|_{\H_r} \|\chi \|_{\H_r}$.

	Let the vector space of polynomials of degree $\leq N-1$ be $E_N$. This set is the image of our interpolation operator $\P_N$, and so we can can computationally represent polynomials in $E_N$ by their coordinate vectors in the basis of Chebyshev-Lagrange polynomials $\{\ell_{n,N}\}_{n=1,\ldots, N}$. The property \eqref{eq:LagrangeProperty} means that this is can be done simply by evaluating at the Chebyshev nodes:
	\[ \vec \psi =  (\psi(x_{j,N}))_{j = 1,\ldots, N}. \]
	Similarly, vectors $\vec v$ represent the polynomials
	\[ v(x) = \sum_{j=1}^N \vec v\j  \ell_{j,N}(x).\]
	where $v\j $ is the $j$th element of $v$.
	
	While our transfer operator is a truly infinite-dimensional operator, we can approximate it on $E_N$ by the restricted, projected operator $\P_N \L_\varphi: E_N \circlearrowleft$. In the Chebyshev-Lagrange basis, this operator is given by an $N\times N$ matrix:
	\begin{equation} L_N = ((\mathcal{L}_\varphi\ell_{k,N})(x_{j,N}))_{j,k = 1,\ldots, N} \label{eq:TransferMatrix}\end{equation}

\section{Computation of weak estimates of equilibrium measures} \label{s:Equilibrium}

		It turns out that under analyticity assumptions on $\varphi_\iota, g_\iota$ and contraction assumptions on $g_\iota$, Chebyshev-Lagrange interpolation is very good at approximating the transfer operator $\L_\varphi$ \cite{Bandtlow20
		}. 

	Now in \eqref{eq:EquilibriumEigenfunction}, our equilibrium measure $\mu_\varphi$ was described as a product of left and right eigenfunctions of the transfer operator. The following result shows that we can simply copy this equation across to the projected operator $L_N$, and get a result whose error is exponentially small in $N$:
	\begin{theorem}\label{t:EquilibriumResult}
	Suppose that for some $R>r>0$, $\| \sum_{\iota \in I} e^{\varphi_\iota} \|_{\H_r} \leq \Phi$ and $g_\iota(\E_r) \subset \E_R$.\footnote{For maps with finite branches, this occurs if the $\varphi_\iota, g_\iota$ are real-analytic and $\sup_{\theta [0,\pi]}|(\cos^{-1} \circ g_\iota\circ \cos)'(\theta)| < 1$ \cite{Wormell19}.}
	
		Let $\vec \nu_N, \vec \h_N \in \mathbb{R}^N$ be leading left and right eigenvectors of the matrix $L_N$, and let 
		\[\vec \mu_N\j = \frac{1}{\sum_{i=1}^N \vec \nu_N^i  \vec \h_N^i }\vec \nu_N\j \vec h_N\j.\]
		 Then there exists $K'$ such that for $N$ sufficiently large and all $\psi \in \H_R$,
	\begin{equation} \left|\int \psi\,\d\mu_\varphi - \sum_{j=1}^N \vec \mu_N^j \psi(x_{N,j}) \right| \leq K'  e^{-(R-r)N} \| \psi \|_{\H_R} \label{eq:EquilibriumResult} \end{equation}
\end{theorem}

The computational algorithm to perform this estimate is thus as follows:
\begin{enumerate}
	\item Construct the $N\times N$ transfer operator matrix $L_N$ according to \eqref{eq:TransferMatrix}.
	\item Estimate the leading left and right eigenvectors of $L_N$, and multiply them together and normalise to obtain $\vec \mu_N$.
	\item Evaluate $\psi$ at the Chebyshev nodes and compute the dot product of this vector with $\vec \mu_N$ \eqref{eq:EquilibriumResult}.
\end{enumerate}
For transfer operators with a reasonable (e.g. finite) number of branches, these steps can be done in a few lines of code. The complexity of this algorithm is $\mathcal{O}(N^2)$, and so $N=1000$ (i.e. hundreds of digits of accuracy) can be handled on a personal computer in a few seconds. 

A plot of convergence of estimates is given in Figure~\ref{f:Convergence}.

Computation of integrals with respect to the conformal measure is also possible: in this case, one would compute $\int \psi \d\nu_\varphi \approx \tfrac{1}{\sum_j h_N\j} \sum_{j=1}^N h_N\j \psi(x_{N,j})$.

\begin{figure}
	\centering
	\includegraphics{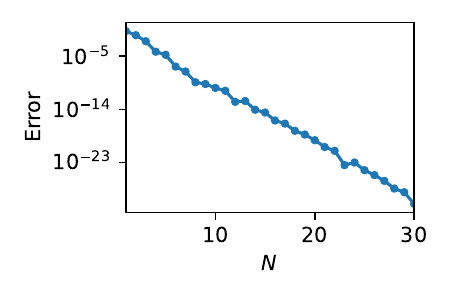}
	\caption{Convergence of the estimate of the equilibrium measure for the Gauss map $f(x) = 1/x \mod 1$ on the continued fraction digits $\{2,3,4,5,6\}$ with respect to the geometric potential $\varphi(x) = \log |f'|$ integrated against $\psi(x) = x$.}
	\label{f:Convergence}
\end{figure}

\section{Markov chain Monte Carlo sampling of equilibrium measures} \label{s:Markov}

One may also wish to obtain representative point samples from an invariant measure $\mu$, for example, orbits $\{ x_t \}_{t = 1,\ldots,T-1}$ where each $x_t \sim \mu$ and each $x_{t+1} = g_{i_t}(x_t)$ for some $i_t \in I$. If one is given a potential $\varphi$ so that 
\begin{equation}\sum_{\iota \in I} e^{\varphi_\iota}(x) = 1\textrm{ for all }x \in [-1,1],\label{eq:EquilibriumMarkov}\end{equation}
then $\mu$ is the invariant measure of the Markov chain on $[-1,1]$ defined by setting $x_{t+1} = g_{i_t}(x_t)$ where the distribution of $i_t$ is given by $\mathbb{P}(i_t = \iota) = e^{\varphi_\iota(x_t)}$. By standard transfer operator theory, the distribution of $x_{t}$ converges exponentially quickly to $\mu$ in most metrics (e.g. Wasserstein distance) regardless of the initialisation of the $x_0$. Thus in practice, a series $\{x_t\} \sim \mu$ can be generated by starting with a randomly initialised $x_{-T_0}$ for some large $T_0$ (e.g. $10,000$). This idea has previously been employed in some restricted circumstances, such as in computing measures of maximal entropy for Markovian logistic maps \cite{Grassberger85}.

However, the condition \eqref{eq:EquilibriumMarkov} implies that the right eigenfunction $\h_\varphi$ must be constant, with pressure $P(\varphi) = 0$, which is obviously non-generic. However, it is standard that a equilibrium invariant measure with respect to $\varphi$ is also a equilibrium invariant measure with respect to the potential $\tilde \varphi = \varphi + w - w \circ f - P(\varphi)$ for any H\"older function $w$, and in particular setting $w = \log \h_\varphi$. In this case, we see that 
\[ L_{\tilde\varphi} \psi(x) = \sum_{\iota \in I} e^{\tilde \varphi(g_\iota(x))} \psi(g_\iota(x)) = e^{-P(\varphi)} \h_\varphi(x)^{-1} \sum_{\iota \in I} e^{\varphi_\iota(x)} (\h_\varphi \psi)(g_\iota(x)) \]1
so constant functions are eigenfunctions of $\L_{\tilde \varphi}$ and the pressure is zero, as required to make a Markov sample of $\mu$. Note that the dynamics (the $g_\iota$) are the same.

Our Chebyshev-Lagrange method allows us to estimate $\h_\varphi$ extremely accurately by polynomials $\h_N(x)$. We can thus generate orbits sampled from $\mu$ as follows:
\begin{enumerate}
	\item Construct the $N\times N$ transfer operator matrix $L_N$ and estimate the leading eigenvalue $e^{P_N}$ and right eigenvector $\vec\h_N$ of $L_N$.
	\item Initialise $x_{-T_0}$ randomly, for $T_0$ large.
	\item Iteratively generate $x_{t+1} = g_{\iota_t}(x_t)$ by choosing $\iota_t \in I$ at each step randomly with probability\footnote{$\L_\varphi \h_N$ is used in the denominator so that all probabilities sum to $1$, but to save computational expense one can replace it $e^{P_N}\h_N(x_t)$ for all branches except one, with this distinguished branch taking the remainder of the probability. All results, in particular Theorem~\ref{t:FourierMC}, go through.}
	\begin{equation} \mathbb{P}(\iota_t = \iota) = \frac{e^{\varphi(g_\iota(x_t))} \h_N(g_\iota(x_t))}{(\L_\varphi \h_N)(x_t)}, \label{eq:IterativePotential} \end{equation}
	where the right eigenfunction estimate is given by $h_N(x) = \sum_{j=1}^N \vec \h_N\j \ell_{j,N}(x)$,
	and setting $x_{t+1} = g_{\iota_t}(x_t)$. Discard iterates up to $x_0$.
\end{enumerate}

Among other things, this is in fact an effective way to obtain almost unbiased estimates of integrals against non-regular functions, such as H\"older functions, or Fourier modes of very high order. The mean of a function $\psi$ can be estimated via a Birkhoff mean:
\[ \int \psi \, \d \mu \approx \frac{1}{T} \sum_{t=1}^T \psi(x_t) =: M_T(\psi). \]
Because the procedure to generate $x_t$ is random, so too is $M_T$. It obeys a central limit theorem, so that asymptotically $M_T(\psi)$ approximates a normal distribution with mean
\[ \mathbb{E}[M_T(\psi)] =  \int \psi \, \d \mu + \mathcal{O}(e^{-(R-r)N}+ T^{-1} e^{-c \alpha T_0})  \]
and standard deviation $\mathcal{O}(T^{-1/2})$. In general, the constants involved are mild for bounded but irregular functions: in fact, as demonstrated in Section~\ref{ss:MCFourier}, they are only $\mathcal{O}(\log |\xi|)$ for Fourier exponentials!

This kind of approximation can be quite effective: setting $T = 10^7$ provides around 3 significant figures of accuracy. Furthermore, the random error can be quantified by taking multiple independent samples and from them generating a Student $t$-test confidence interval for their expectation value. This kind of confidence interval estimate has been used to obtain reliable evidence for conditional mixing, a question related to fractal decay \cite{Wormell23}.

Recalling that $\mu = \nu_\varphi \h_\varphi$ and therefore $\nu_\varphi = \h_\varphi^{-1} \mu$, the conformal measure $\nu_\varphi$ can also be sampled by weighting the $x_t$ that sample $\mu$ by $\h_N^{-1}(x_t)$. So, for example,
\[ \int \psi \, \d \nu_\varphi \approx \tfrac{1}{T} \sum_{t=1}^T \psi(x_t) \h_N^{-1}(x_t). \]

\section{Application to computing Fourier transforms} \label{s:Applications}

\begin{figure}
	\centering
	\includegraphics{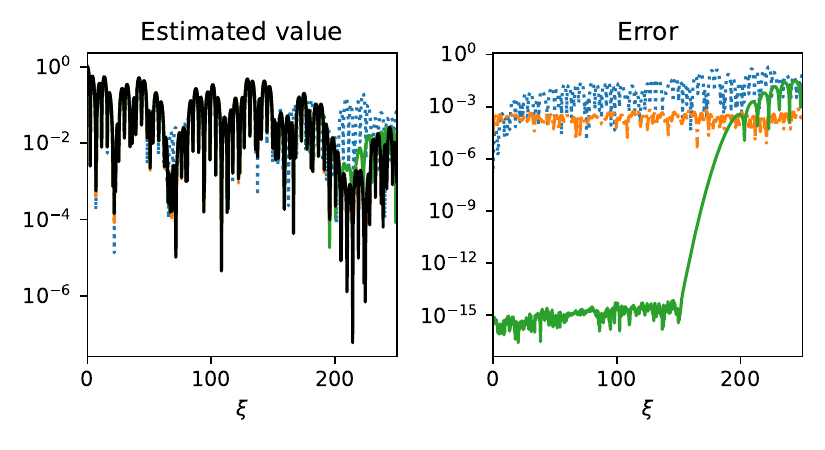}
	\caption{Left: the absolute value of true Fourier transform for the uniform measure of the middle-$1/\pi$ Cantor set on $[-1,1]$ (black), compared against a Chebyshev-Lagrange integral estimate as in Section~\ref{s:Equilibrium} (green line), a Chebyshev-Lagrange-based Monte Carlo estimate with $T = 10^7$ (orange dash-dots), and an Ulam estimate \cite{Froyland07} (blue dots). Right: plot of the absolute value of the respective errors. Discretisations of order $N = 200$ were used. Note that the discrepancy for the integral estimate is exponentially small (much smaller than the standard floating-point round-off error of around $ 10^{-16}$) but only for $\xi \lesssim N$.}
	\label{f:Fourier}
\end{figure}

The Fourier transform of a measure $\mu_\varphi$ is given by
\[ \hat \mu_\varphi(\xi) = \int e^{-i\xi x}\,\d\mu_\varphi(x). \]

The decay of the Fourier transform for fractal measures, and in particular equilibrium measures, has recently become a hot topic in fractal geometry \cite{Sahlsten20,Leclerc23, Sahlsten23}. There are many open questions in this area (see the following survey \cite{Sahlsten23}), but methods in this note can already be used very effectively to numerically estimate the Fourier transform of equilibrium measures.

\subsection{Chebyshev-Legendre approximation}

As complex exponentials are smooth, our weak approximation algorithm can be used very effectively. However, some care must be taken that the approximation's resolution is not worse than the exponential function's oscillatory scale:
\begin{corollary}\label{c:FourierCL}
	Under the previous assumptions, there exists $K'$ such that for $N$ sufficiently large and all $\psi \in \H_R$,
	\begin{equation*} \left|\hat\mu_\varphi(\xi) - \sum_{j=1}^N \vec \mu_N^j e^{i \xi x_{N,j}} \right| \leq K'  e^{-R((1-\tfrac{r}{R})N-\xi)} \end{equation*}
\end{corollary}
\begin{proof}
	We apply Theorem~\ref{t:EquilibriumResult} with $\psi(x) = e^{i \xi x}$. In this case $\| \psi \|_{H_R}  = \sup_{z \in H_R} |e^{i \xi z}| = e^{\xi \sinh R}\leq e^{\xi R}$. The result follows.
\end{proof}
This result implies that we get exponentially good results provided that $\xi \lesssim c N$, where
 \[c = \frac{R-r}{\sinh R} \approx 1 - \sup_{\substack{\iota \in I\\ \theta \in [0,\pi]}} |(\cos^{-1} \circ g_\iota \circ \cos)'(\theta)| < 1.\]
 In practice, however it appears to be acceptable for $\xi \lesssim N$ (see Figure~\ref{f:Fourier}). An example of this application is given in Figure 1 of \cite{Sahlsten23}.
 
 \subsection{Monte Carlo estimate}\label{ss:MCFourier}
 
 On the other hand, it is also possible to use Monte Carlo approximations like those in in Section~\ref{s:Markov} to obtain an approximation of the Fourier transform. This will be less accurate, due to the slow central limit theorem-type convergence, but, it turns out, for the weak approximation in Theorem~\ref{t:EquilibriumResult}, we find that Monte Carlo method works well at frequencies $\xi$ much, much larger than the notional resolution of the transfer operator:
\begin{theorem}\label{t:FourierMC}
	Let $\ell_\xi = 1 + \log (1 + |\xi|)$. There exist $c<1$ and $K'', P_1, P_2, P_3$ such that for every $\xi \in \mathbb{R}$ and any initialisation $x_{-T_0}$:
	\begin{enumerate}[a.]
		\item The error of the expectation of $M_T$ is exponentially small when $T_0 \gg \ell_\xi$ and $N \gg \log \ell_\xi$.
	\[ \left| \bE M_T(e^{i\xi \cdot}) - \hat\mu_\varphi(\xi) \right|  \leq K'' \left(T^{-1} c^{ T_0 - \ell_\xi} + e^{-(R-r)N} \ell_\xi \right), \]
	\item $M_T$ obeys a Central Limit Theorem in $T$. Indeed, when $T$ is greater than a fixed constant $T'$, for any $p \geq e^{-P_1 (T/T' - 1)}$ we have with probability $1- p$,
\begin{equation} \left| M_T(e^{i\xi \cdot}) - \bE M_T(e^{i\xi \cdot}) \right|  \leq \frac{\sqrt{P_2 + P_3 \log p^{-1}}}{\sqrt{T/\ell_\xi^2}}. \label{eq:CLTbound}\end{equation}
	\end{enumerate}
\end{theorem}
The proof of this result is given in Appendix~\ref{s:FMCProof}. The success of the Monte Carlo algorithm at high frequencies is driven by the fact that what spatially localises the measure $\mu_\varphi$ is the IFS $\{g_\iota\}$, which, unlike the eigenfunctions, we know exactly. The dependence on the IFS weights is more or less down to pointwise convergence, plus an initial mixing time of $\mathcal{O}(\log |\xi|)$ timesteps.
	
An illustrative example of the power of Theorem~\ref{t:FourierMC} is that given an IFS that contracts the complex disc of radius $1.5$, if $N=10$ and $T = 10^9$, the Fourier transform of the equilibrium measure on $[-1,1]$ can be estimated to around 3 digits for frequencies as large as $\xi \sim 10^{12}$ (at which point high-precision floating point would become necessary). Furthermore, the Gaussian-esque relationship between the tail probability $p$ and the error size in \eqref{eq:CLTbound} means that large numbers frequencies $\xi$ can be estimated from the same point sample $\{x_t\}_{t = 1,\ldots,T}$ without overly large outliers in the Fourier transform being generated.

For pictorial illustration, we give a comparison of the Monte Carlo method against the true Fourier transform at high frequencies $\xi$ in Figure \ref{f:highfreq}.

\begin{figure}
	\centering
	\includegraphics{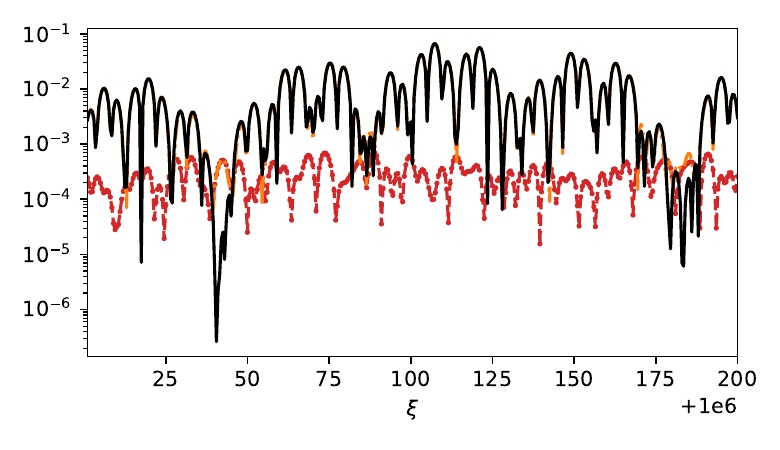}
	\caption{For $\xi$ between $1,\!000,\!000$ and $1,\!000,\!200$, the absolute value of true Fourier transform for the uniform measure of the middle-$1/\pi$ Cantor set on $[-1,1]$ (black) compared against a Monte Carlo estimate with $T = 10^7$ (orange dash-dots), and the error between the two (red dashed line).}
	\label{f:highfreq}
\end{figure}

\section{Extensions}\label{s:Extensions}

The computational techniques suggested in this paper, and in particular the meta-formula (generalising \eqref{eq:EquilibriumResult}) that
\[ \int \psi\,\d\mu \approx \sum_{j=1}^N \frac{\vec \nu_N \cdot \overrightarrow{\P_N (\h_N \psi)} }{\vec \nu_N \cdot \vec \h_N}\]
is very general, and can be applied to other kinds of systems and transfer operator discretisations. 

Further examples where a Chebyshev-Lagrange type discretisation can be useful are in multiple-dimensional iterated function systems \cite{Vytnova23} and infinite-branched transfer operators \cite{Wormell21,Vytnova23,Macleod93}, as well as in hyperbolic dynamics.

Other kinds of discretisations will also work effectively here. It is already known that it can be applied to Ulam's method, which approximates functions by piecewise constant functions \cite{Froyland07}, although the convergence rate of Ulam's method is slow, at $\mathcal{O}(N^{-1}) \log N$. Furthermore, estimates used using Ulam schemes cannot effectively be used to obtain good weak estimates of Fourier transforms, as even for uniformly expanding maps the error can be expected to be $\mathcal{O}(\xi N^{-1} \log N)$ (see Figure~\ref{f:Fourier}). This means Ulam method requires $N/\xi \gg 1$ for good estimates, unlike the analytic Chebyshev-Lagrange case which only requires $N-\xi\gg 1$. However, it may potentially be possible to use Ulam's methods estimates for the eigenfunction to obtain reasonable Monte Carlo estimates using methods in Section~\ref{s:Markov}. Other discretisations that may be profitably used to compute properties of equilibrium measures include Extended Dynamic Mode Decomposition, where transfer operators may be generated from orbits of the system, with no knowledge of the map required \cite{Bandtlow23}, as well as higher-order bases that generate sparse transfer operator matrices, which may be the most numerically efficient technique to study finitely differentiable maps \cite{Ding93}.



\appendix
\section{Proof of Theorem~\ref{t:EquilibriumResult}}\label{s:CLProof}
	
	In this appendix we will prove Theorem~\ref{t:EquilibriumResult}.
	
		Under our assumptions, the transfer operator maps $\H_r$ into a space of higher regularity $\H_R,\, R>r$ \cite[Lemma~3.2]{Bandtlow20}. On the other hand, the interpolation $\P_N$ is very close in norm to the identity when considered as a map $\H_R \to \H_r$ \cite[Lemma~2.6]{Bandtlow20}. The synthesis of these results, based on \cite[Theorem~3.3]{Bandtlow20}, is as follows:
	\begin{proposition}\label{p:SpectralError}
		Suppose that for some $R>r>0$, $\| \sum_{\iota \in I} e^{\varphi_\iota} \|_{\H_r} \leq \Phi$ and $g_\iota(\E_r) \subset \E_R$. Then $\L_\varphi$ is a compact endomorphism on $\H_R$, and there exists a constant $C$ so that for all $N \in \mathbb{N}$,
		\[ \| \P_N \L - \L \|_{\H_r} \leq C \Psi e^{(R-r) N}.\]
		
		Hence, there exists $K$ such that for $N$ large enough, if $e^{P_N(\varphi)}$ is the simple leading eigenvalue of $\P_N \L$ with a corresponding left (resp. right) eigenvector $\nu_N \in \H_r^*$ (resp. $\h_N \in \H_r$), then
		\[ |P_N(\varphi) - P(\varphi)|, \|\nu_N - \nu_\varphi\|_{\H_r^*}, \|\h_N - \h_\varphi\|_{\H_r} \leq K e^{-(R-r)}. \]
	\end{proposition}
	
	This result allows us to study equilibrium measures using $\P_N \L$:
	\begin{proposition}\label{p:EquilibriumIntegrals}
		Suppose $\psi \in \H_R$. Then there exists $K'$ such that for $N$ sufficiently large,
		\[ \left| \frac{\nu_N[\P_N[\psi \h_N]]}{\nu_N[\h_N]} - \int \psi\,\d\mu_\varphi \right| \leq K'  e^{-(R-r)N} \| \psi \|_{\H_R} \]
	\end{proposition}
	\begin{proof}
		From \eqref{eq:EquilibriumEigenfunction} we are interested in the error 
		\[ \left| \frac{\nu_N[\P_N[\psi \h_N]}{\nu_N[\h_N]} - \frac{\nu_\varphi[\psi \h_\varphi]}{\nu_\varphi[\h_\varphi]} \right|. \]
		
		We can assume without loss of generality that $\h_\varphi$ is of unit $H_r$ norm, and $\nu_\varphi[\h_\varphi] = 1$, so 
		$\|\h_N\|_{\H_r} - 1,   \| \nu_N \|_{\H_r^*} - \|\nu_\varphi\|_{\H_r^*}, \nu_\varphi[\h_\varphi] - 1$ are all of order $e^{-(R-r)N}$. 
		
		Here, we have that
		\begin{align*} \P_N[\psi \h_N] - \psi \h_N &= (I - \P_N) [\psi e^{-P_N(\varphi)} \P_N\L \h_N]\\
			&= e^{-P_N(\varphi)} \big((I-\P_N)[\psi (I-\P_N) \L \h_N] + (I-\P_N)[\psi \L \h_N] \big)
		\end{align*}
		Let $s = (r+R)/2$. Using Lemmas~2.9 and~3.2 in \cite{Bandtlow20}, we can bound the norms by
		\begin{align*}
			\|\P_N[\psi \h_N] - \psi \h_N \|_{\H_r} &\leq e^{-\Re P_N(\varphi)} \Big( \|I-\P_N\|_{\H_{s} \to \H_{r}} \|\psi\|_{\H_R} \|I-\P_N\|_{\H_R \to \H_{s}}\\
				&\qquad + \| I-\P_N \|_{\H_R \to \H_r}\|\psi\|_{\H_R} \Big) \|\L\|_{\H_r \to \H_R} \|\h_N\|_{\H_r}\\
				&\leq C' \left( C' e^{-N(s-R)} \| \psi \|_{\H_R} C' e^{-N(R-r)/2} + C' e^{-N(R-s)} \| \psi\|_{\H_R} \right) C' \\
				&\leq K' e^{-N(R-r)} \|\psi\|_{\H_R}
			\end{align*}
			As a consequence, and since $\nu_N$ is bounded in the dual of $\H_r$, we only need to appropriately bound
			\[ \left| \frac{\nu_N[\psi \h_N]}{\nu_N[\h_N]} - \frac{\nu_\varphi[\psi \h_\varphi]}{\nu_\varphi[\h_\varphi]} \right|, \]
			which obtains from simple applications of Proposition~\ref{p:SpectralError}.
		\end{proof}
		
		It now remains to translate Proposition~\ref{p:EquilibriumIntegrals} into the terms of the vectors used in computation.
		
		All right eigenfunctions of $\P_N \L$ which correspond to non-zero eigenvalues must lie in $E_N$ (i.e. be polynomials of degree at most $N-1$). We thus recover our right eigenfunction $\h_N$ of $\L_N$ as a right eigenvector of $L_N$:
		\[ \vec r = (\h_N(x_{j,N}))_{j = 1,\ldots, N}. \]
		
		We might ask ourselves how we compute our left eigenvector. Let us construct the row vector
		\[ \vec \nu_N = (\nu_N[\ell_{k,N}])_{k = 1,\ldots, N}. \]
		Then,
		\begin{align*} (\vec \nu_N L)_k &= \sum_{j=1}^N \nu_N[\ell_{j,N}] (\mathcal{L}_\varphi\ell_{k,N})(x_{j,N})\\
			&=  \sum_{j=1}^N \nu_N\left[\ell_{j,N} (\mathcal{L}_\varphi\ell_{k,N})(x_{j,N}) \right]\\
			&= \nu_N \left[\P_N \mathcal{L}_\varphi \ell_{k,N} \right]\\
			&= e^{P_N} \nu_N[\ell_{k,N}] = e^{P_N} (\vec \nu_N L)_k
		\end{align*}
		so $\vec \nu_N$ is the leading left eigenvector of $L$!
	
		\begin{proof}[Proof of Theorem~\ref{t:EquilibriumResult}]
			From the preceding discussion, we can write $\nu_N[v] = \vec \nu_N \cdot (v(x_{j,N}))_j$ for a function $v \in E_N$. We have that $\h_N, \P_N[\psi \h_N] \in E_N$, and
			\[ \P_N[\psi \h_N](x_{j,N}) = (\psi \h_N)(x_{j,N}) = \psi(x_{j,N})\, \vec \h_N\j . \]
			This combined with Proposition~\ref{p:EquilibriumIntegrals} gives us the result.
		\end{proof}
	
	\section{Proof of Theorem~\ref{t:FourierMC}}\label{s:FMCProof}
	We now prove Theorem~\ref{t:FourierMC}. This result is built on standard dynamics techniques, but it it involves keeping careful track of the Fourier frequency $\xi$.
	
	\begin{proof}[Proof of Theorem~\ref{t:FourierMC}]
		Let us notate the Chebyshev-discretised potential implicitly used in the Monte Carlo algorithm \eqref{eq:IterativePotential} as
		\[ \tilde\varphi_{\iota,N} = \varphi_\iota(x) + \log \h_N(g_\iota(x)) - \log (\L_\varphi \h_N)(x). \]
		
		
		Let the centred observable be $\tilde e_\xi(y) := e^{i\xi y} - \hat\mu_{\tilde\varphi_N}(\xi)$.
		
		We decompose the deterministic part into two parts:
		\[ \bE M_T(e^{i\xi \cdot}) -  \hat\mu_\varphi(\xi) = \left(\bE M_T(e^{i\xi \cdot}) - \hat\mu_{\tilde\varphi_N}(\xi)\right) + 
		\left(\hat\mu_{\tilde\varphi_N}(\xi) - \hat\mu_\varphi(\xi)\right). \]
		Let the distribution of $x_{-T_0}$ be $\nu$. The first part can be written as a bound using the random process' Koopman operator $\L_{\tilde \varphi_N}$:
		\begin{align*} \bE M_T(e^{i\xi \cdot}) - \hat\mu_{\tilde\varphi_N}(\xi) &=  T^{-1} \sum_{t=1}^T \int\bE[\tilde e_\xi(x_t)] = T^{-1} \sum_{t=1}^T \int \L_{\tilde\varphi_N}^{t+T_0}[\tilde e_\xi] \d\nu \\
		\end{align*}
		Now, $\L_{\tilde\varphi_N}$ has a spectral gap in $\Lip$, the set of bounded Lipschitz functions on the interval, with leading eigenvalue $1$ and left eigendistribution $\mu_{\tilde\varphi_N}$. Since $\int e_\xi\, \d \mu_{\tilde\varphi_N} = 0$, this means that 
		\[ | \bE M_T(e^{i\xi \cdot}) - \hat\mu_{\tilde\varphi_N}(\xi)| \leq C c^{t+T_0} |\xi| \]
		for some constant $C$, so
		\[ |\bE M_T(e^{i\xi \cdot}) - \hat\mu_{\tilde\varphi_N}(\xi)| \leq K'' T^{-1} c^{T_0} |\xi|, \]
		as required.
		\\
		
		Using the same spectral gap type of argument for the second part of the deterministic error, we have that
		\begin{align} \left|\hat\mu_{\tilde \varphi_N}(\xi) - \hat\mu_{\tilde\varphi}(\xi)\right| &=\lim_{T\to\infty} \left| \int (\L_{\tilde\varphi_N}^T - \L_{\tilde\varphi}^T) \tilde e_\xi\,\d \mu_{\tilde\varphi} \right| \notag\\
		&\leq \lim_{T\to\infty} \sum_{t=0}^{T-1}  \left| \int \L_{\tilde\varphi}^{T-t-1} (\L_{\tilde\varphi_N} - \L_{\tilde\varphi})  \L_{\tilde\varphi_N}^{t} \tilde e_\xi \,\d \mu_{\tilde\varphi} \right| \notag\\
		&=  \sum_{t=0}^\infty  \left| \int (\L_{\tilde\varphi_N} - \L_{\tilde\varphi})  \L_{\tilde\varphi_N}^t \tilde e_\xi\,\d \mu_{\tilde\varphi} \right| \label{eq:DifferenceSum}
		\end{align}
		using that $\mu_{\tilde\varphi}$ is invariant under $\L_{\tilde\varphi_N}^*$ in the last line. We have for $\psi \in C^0$ that
		\[ |(\L_{\tilde\varphi_N} - \L_{\tilde\varphi}) \psi| = \left|\sum_{\iota \in I} \left((e^{\tilde \varphi} - e^{\tilde \varphi_N})\psi\right)\circ g_\iota \right| \leq \|e^{\tilde \varphi} - e^{\tilde \varphi_N}\|_{L^\infty} \|\psi\|_{C^0} \leq C e^{-(R-r)N} \|\psi\|_{C^0}, \]
		by Proposition~\ref{p:SpectralError}.
		
		We substitute in $\psi = \L_{\tilde\varphi_N}^{t} \tilde e_\xi$. The $C^0$ norm of this has two alternative bounds
		\[ \|  \L_{\tilde\varphi_N}^{t} \tilde e_\xi \|_{C^0} \leq \|  \L_{\tilde\varphi_N}^{t} \tilde e_\xi \|_{\Lip} \leq C c^t \| \tilde e_\xi \|_{\Lip} \leq C c^t (1 + |\xi|)  \]
		and
		\[ \|  \L_{\tilde\varphi_N}^{t} \tilde e_\xi \|_{C^0} \leq \| \tilde e_\xi \|_{C^0} \leq 2. \]
		Combining these into \eqref{eq:DifferenceSum}, we get the required bound.
		\\
		
		For the deviation from expectation values, we can break up $e^{i\xi x}$ into real and imaginary parts. These have central limit theorems by standard results \cite{Hennion01}. To obtain the explicit bounds on the error, we can make an upper and lower concentration of measure bound on $M_T$ on each of them\cite{Chazottes15}. Let $f_\xi(x) = \cos \xi x - \int \cos \xi \cdot \d\mu_{\tilde\varphi_N}$. Then for any $b, \omega > 0$, we have
		\begin{align}
		\bP\left(M_T(\cos\xi \cdot) - \Re\hat\mu_{\tilde\varphi_N}(\xi) > b \right) &= \bP\left(e^{ T M_T(\omega f_\xi)} > e^{T \omega b} \right)\notag \\
		&= \leq e^{-T\omega b} \bE[e^{ T M_T(\omega f_\xi)}]\notag \\
		\leq e^{-T b\omega} \int \L_{\tilde\varphi_N}^{T_0} \L_{\tilde\varphi_N + \omega f_\xi}^T 1\, \d\nu, \label{eq:ExponentialBound} \end{align}
		where we used Markov's inequality followed by studying $\L_{\tilde\varphi_N}$ as the Koopman operator of the process $\{x_t\}$. We know that $\| \L_{\tilde\varphi_N}^{T_0} \|_{C^0} = 1$, so to get our concentration bound we need to bound
		\[ \| \L_{\tilde\varphi_N + \omega f_\xi}^T \|_{C^0}, \]
		bearing in mind that $\omega f_\xi$ oscillates with frequency $\xi$. For concision we notate the potential $\tilde\varphi_N + \omega f_\xi = \chi_\omega$. We will assume that $\omega$ is small, say less than $1$.
		
		Fix a constant $\omega_0$ and choose sufficiently large $M, M'$ independent of $\xi$. Then one can show that for $\omega < \in [0,\omega_0\ell_\xi^{-1}]$, $\L_{\chi_\omega}$ leaves invariant the cone of positive functions $\psi$ such that for all $x,y$,
		\begin{equation} \left|\log \frac{\psi(x)}{\psi(y)} \right| \leq \Omega_\xi(|x-y|) := (1+M') M |x-y| + M' \frac{\log (1 + \xi|x-y|)}{\ell_\xi}. \label{eq:ModulusofContinuity}\end{equation}
		In particular, the leading $\Lip$-eigenfunction $h_{\chi_\omega}$ of $\L_{\chi_\omega}$ is contained in this cone. If $h_{\chi_\omega}$ is normalised to have supremum $1$, \eqref{eq:ModulusofContinuity} gives us that for all $y$,
		\[-\log h_{\chi_\omega}(y) \leq \Omega_\xi(2) = 2(1+M')M - M' \ell_\xi^{-1} \log (1+2\xi) \leq M'' \]
		where $M'' = 2(1+M')M + M'$. In other words,
		\[ 1 \leq e^{M''} h_{\chi_\omega}.  \]
		Now, applying $\L_{\chi_\omega}^T$ to both sides, we have 
		\begin{equation} \L_{\chi_\omega}^T 1 \leq e^{T P(\chi_\omega) + M''} h_{\chi_\omega}. \label{eq:CouplingPointwise}\end{equation}
		
		It remains to figure out what $ P(\chi_\omega)$ is. We have that $P(\chi_0) = P(\tilde \varphi_N) = 0$, and since $f_\xi$ is mean-zero with respect to $\mu_{\tilde \varphi_N}$, $\tfrac{\d}{\d\omega} P(\chi_\omega)$ is zero at $\omega = 0$.
		
		Recall the notation $\ell_\xi := 1 + \log (1 + |\xi|)$. Now, we can apply an explicit coupling argument to $\L_{\chi_\omega}^{\lceil \ell_\xi \rceil}$ \`a la \cite{Kosloff19} using the modulus of continuity $\Omega_\xi$ in \eqref{eq:ModulusofContinuity}. This gives that for all $\psi \in \Lip$,
		\[ \left\| e^{-tP(\chi_\omega)} \L_{\chi_\omega}^t \psi - h_{\chi_\omega} \int \psi\,\d\nu_{\chi_\omega} \right\|_{\mathcal{B}_\xi} \leq C c^{t/\lceil \ell_\xi \rceil} \| \psi \\_{\mathcal{B}_\xi} \]
		where $C, c$ are uniform in $\xi$, and 
		\[ \| \psi \|_{\mathcal{B}_\xi} := \|\psi\|_{C^0} + \sup_{x,y} \frac{|\psi(x) - \psi(y)|}{\Omega_{\xi}(|x-y|)}. \]
		This means we have a kind of uniform decay of correlations where the mixing time slows down as $\mathcal{O}(\ell_\xi)$.
		 
		Making perturbation expansions of $P(\chi_\omega)$ in $\omega$ \cite{Crimmins19}, and using that $f_\xi, f_\xi^2$ have $\B_\xi$ norm of $\mathcal{O}(\ell_\xi)$, we can then uniformly bound the second derivative of $P(\chi_\omega)$ as $\mathcal{O}(\ell_\xi^2)$, giving constants $Q, \omega_0$ such that when $\omega \in [0, \omega_0 \ell_\xi^{-1}]$, 
		\[ P(\chi_\omega) \leq Q \ell_\xi^2 \omega^2. \]
		So, combining this with \eqref{eq:ExponentialBound} and \eqref{eq:CouplingPointwise}, we have that 
		\[ \bP\left(M_T(\cos\xi \cdot) - \Re\hat\mu_{\tilde\varphi_N}(\xi) > b \right) \leq e^{Q T \ell_\xi^2 \omega^2 - Tb\omega + M''}. \]
		If $b \leq 2Q w_0 \ell_\xi$, we can set $\omega = b/(2Q\ell_\xi^2) \in [0,\omega_0 \ell_\xi^{-1}]$ and obtain that 
		\[\bP\left(M_T(\cos\xi \cdot) - \Re\hat\mu_{\tilde\varphi_N}(\xi) > b \right) \leq e^{M'' - T b^2/4Q\ell_\xi^2}. \]
		We can then consider the lower bound, and the bound for the imaginary (sine) part. Since by the proof of part a. the difference between $\bE M_T(e^{i\xi\cdot})$ and $\hat\mu_{\tilde\varphi_N}(\xi)$ is $\mathcal{O}(T^{-1})$ for $T \gg \ell_\xi$, we can relax the constants a little to get that
				\[\bP\left(|M_T(e^{i\xi \cdot} - \hat\mu_{\tilde\varphi_N}(\xi)| > 2b \right) \leq 2e^{M'' - T b^2/4Q\ell_\xi^2} =: p. \]
		The equation (and restrictions on $b$) can then be rewritten to get a result in terms of $p$.
	\end{proof}
	
\bibliographystyle{plain}
\bibliography{gibbs}

\end{document}